\documentclass[12pt]{amsart}

\theoremstyle{plain}
 \newtheorem{theorem}{Theorem}[section]

\newtheorem{lemma}[theorem]{Lemma}
\newtheorem{corollary}[theorem]{Corollary}
\newtheorem{proposition}[theorem]{Proposition}
\newtheorem{conj}[theorem]{Conjecture}

\theoremstyle{definition}
 \newtheorem{remark}[theorem]{Remark}
\newtheorem{definition}[theorem]{Definition}

\newcommand{\bbC}{\mathbb{C}}
\newcommand{\bbN}{\mathbb{N}}
\newcommand{\bbR}{\mathbb{R}}
\newcommand{\bbZ}{\mathbb{Z}}
\newcommand{\bbE}{\mathbb{E}}
\newcommand{\bbS}{\mathbb{S}}
\newcommand{\bbH}{\mathbb{H}}

\newcommand{\Vol}{\mathrm{Vol}}
\newcommand{\intU}{\accentset{\circ}{U}}
\newcommand{\bfi}{\mathbf{i}}
\newcommand{\bfr}{\mathbf{r}}
\newcommand{\Imm}{\mathrm{Im}\, }

\usepackage{accents}
\usepackage{amsmath,amsfonts,amssymb,amscd,amsthm,mathrsfs,bbm}
\usepackage{latexsym}
\usepackage{graphicx}
\usepackage{enumitem}

\begin{document}

\title[The central set and the Kneser-Poulsen conjecture]{The central set and its application to the Kneser-Poulsen conjecture}

\author{Igors Gorbovickis}

\address{Jacobs University Bremen, Campus Ring 1, 28759 Bremen, Germany.}
\email{i.gorbovickis@jacobs-university.de}

\begin{abstract}
The Kneser-Poulsen conjecture says that if a finite collection of balls in a Euclidean (spherical or hyperbolic) space is rearranged so that the distance between each pair of centers does not increase, then the volume of the union of these balls does not increase as well. 
We give new results about central sets of subsets of a Riemannian manifold and apply these results to prove new special cases of the Kneser-Poulsen conjecture in the two-dimensional sphere and the hyperbolic plane.

51M16 (51M25, 51M10, 53C20, 53C22)
\end{abstract}

\maketitle

\section{Introduction}\label{sec:intro}

\subsection{Kneser-Poulsen conjecture}
The longstanding conjecture independently stated by Kneser~\cite{Kneser} in 1955 and Poulsen~\cite{Poulsen} in 1954 says that 
\begin{conj}\label{KP_conj}
If a finite set of (not necessarily congruent) balls in an $n$-dimensional Euclidean space is rearranged so that the distance between each pair of centers does not increase (we call this a contractive rearrangement), then the volume of the union of the balls does not increase as well. 
\end{conj}

The conjecture was proved by Bezdek and Connelly~\cite{BezdekConnelly} for $n=2$ and remains wide open for $n\ge 3$. Results confirming the conjecture and its generalizations in various special cases, were obtained by Bollob\'as~\cite{Bollobas}, Bern and Sahai~\cite{Bern_Sahai}, Csik\'os~\cite{Csikos95, Csikos98,Csikos01,Csikos_2006}, Gromov~\cite{Gromov}, Gordon and Meyer~\cite{Gordon_Meyer}, Bezdek and Connelly~\cite{BezdekConnelly, BezCon_Sphere} and the author~\cite{IG_6,IG_5}. 

The proofs of the above mentioned results are essentially based on the consideration of continuous contractions, where by a continuous contraction we mean a continuous motion of centers of the balls, such that pairwise distances between them decrease weakly monotonically. It follows from Csik\'os' formula~\cite{Csikos98} that the volume of the union of the balls also decreases weakly monotonically along such motions. The main difficulty of this approach is the fact that not every contractive rearrangement can be implemented as a continuous contraction within the same phase space. As a possible solution to this problem, Bezdek and Connelly in~\cite{BezdekConnelly} suggested to consider a continuous contraction of the centers in a larger ambient space and instead of the original Csik\'os' formula to apply its suitable modification. They obtained an analog of Csik\'os' formula, which was sufficient to prove the Kneser-Poulsen conjecture for $n=2$. Other suitable modifications of Csik\'os' formula were obtained in~\cite{BezCon_Sphere},~\cite{IG_5} and~\cite{IG_6}, which provided a proof of the Kneser-Poulsen conjecture for an arbitrary dimension $n$, but under quite strong additional assumptions on the ball configurations. 


Finally, we notice that the Kneser-Poulsen conjecture can also be formulated for configurations of balls in $n$-dimensional spherical and hyperbolic spaces instead of the Euclidean spaces. The spherical and hyperbolic versions of the conjecture are currently open even for $n=2$.




\subsection{Main results}

In this paper we introduce a new approach to the Kneser-Poulsen conjecture. 
Our approach is based on new results about central sets and about \textit{relative central sets} which we define in Section~\ref{rel_cut_loc_sec}. 

Applying our method, we prove the following new special case of the Kneser-Poulsen conjecture in the 2-dimensional sphere and the hyperbolic plane:

\begin{theorem}\label{main_theorem}
If the union of a finite set of (not necessarily congruent) closed disks in $\bbS^2$ or $\bbH^2$ has a simply connected interior, then the area of the union of these disks cannot increase after any contractive rearrangement. 
\end{theorem}

We hope that further development of the approach suggested by this paper, can improve currently known results about the Kneser-Poulsen conjecture not just in 2-dimensional spaces, but also in spaces of arbitrary dimension. We describe our approach in Section~\ref{method_sec}. The proofs of all lemmas stated in Section~\ref{method_sec}, are given in Section~\ref{cut_locus_sec}.

It is worth mentioning that in contrast to the previously known results, the proof of Theorem~\ref{main_theorem} is not based on the consideration of a continuous contraction. Moreover, the proof of Theorem~\ref{main_theorem} for the cases of the sphere and the hyperbolic plane (as well as the Euclidean plane) is the same, which provides another evidence for the hypothesis that if the general Kneser-Poulsen conjecture holds in the Euclidean spaces, it should also hold in the spherical and hyperbolic spaces. 
At the same time, according to~\cite{Csikos_Kunszenti}, there are no other complete connected Riemannian manifolds in which the Kneser-Poulsen conjecture can possibly be true.

\subsection{Corollaries}



In case of the sphere $\bbS^2$, Theorem~\ref{main_theorem} can be equivalently reformulated in terms of the intersections of disks instead of the unions:

\begin{corollary}\label{intersect_corollary}
	If the intersection of a finite set of (not necessarily congruent) closed disks in $\bbS^2$ is connected then the area of the intersection of these disks cannot decrease after any contractive rearrangement.
\end{corollary}
\begin{proof}
	The proof follows from Theorem~\ref{main_theorem} applied to each connected component of the complement of the intersection.
\end{proof}

Since disks of radii not greater than $\pi/2$ are convex in the unit sphere $\bbS^2$, so is their intersection. Hence, their intersection is connected, and we obtain the following:

\begin{corollary}\label{large_disk_corollary}
(i) If a finite set of disks in $\bbS^2$ with radii not smaller than $\pi/2$ is rearranged so that the distance between each pair of centers does not increase, then the area of the union of the disks does not increase.

(ii) If a finite set of disks in $\bbS^2$ with radii not greater than $\pi/2$ is rearranged so that the distance between each pair of centers does not increase, then the area of the intersection of the disks does not decrease.
\end{corollary}

Corollary~\ref{large_disk_corollary} improves the two-dimensional version of the theorem of Bezdek and Connelly~\cite{BezCon_Sphere}, in which they prove the same statement but for equal disks of radius exactly $\pi/2$.

Finally, we notice that our methods provide a new proof of the planar Kneser-Poulsen conjecture for intersections of disks. This result was originally proved 
in~\cite{BezdekConnelly}.

\begin{corollary}\cite{BezdekConnelly}
If a finite set of (not necessarily congruent) disks in the Euclidean plane is rearranged so that the distance between each pair of centers does not increase, then the area of the intersection of the disks does not decrease. 
\end{corollary}
\begin{proof}
Assume that there exists a counterexample to the statement of the corollary. Then by continuity there exists a counterexample to the similar statement for the disks on a sphere of a sufficiently large radius. Moreover, by continuity we can make sure that the intersection of the disks in the initial configuration on the sphere is connected. 
The latter contradicts to the statement of Corollary~\ref{intersect_corollary}.
\end{proof}

\subsection*{Acknowledgment.} The author would like to thank Robert Connelly for numerous stimulating discussions and Serge Tabachnikov for providing some useful references. The author is also grateful to an anonymous referee for his/her valuable comments.

\section{Description of the method}\label{method_sec}
In this section we describe the proposed general approach to the Kneser-Poulsen conjecture. Most of the lemmas, formulated here, will be proved later in Section~\ref{cut_locus_sec}. These lemmas will follow from more general results about central sets of compact subsets in a complete connected Riemannian manifold. In the end of this section we will give a proof of Theorem~\ref{main_theorem} modulo the formulated lemmas.

In the following discussion, let $M$ be one of the three spaces $\bbS^n$, $\bbE^n$ or $\bbH^n$, where $n\ge 2$, although some of the statements will also hold in a more general setting, where $M$ is an arbitrary complete connected Riemannian manifold.

Let $U\subsetneq M$ be a compact subset of $M$. We will say that a closed ball $B\subset U$ (possibly of zero radius) is \textit{maximal in} $U$, if it is not a proper subset of any other ball $B'\subset U$. The set $C_U\subset U$ that consists of the centers of all maximal balls, is called \textit{the central set} of $U$ (see Figure~\ref{pic} for an example of a central set). We note that typically, the set $C_U$ is infinite.


\begin{figure}[t]
\includegraphics[width=.43\textwidth, angle=-90]{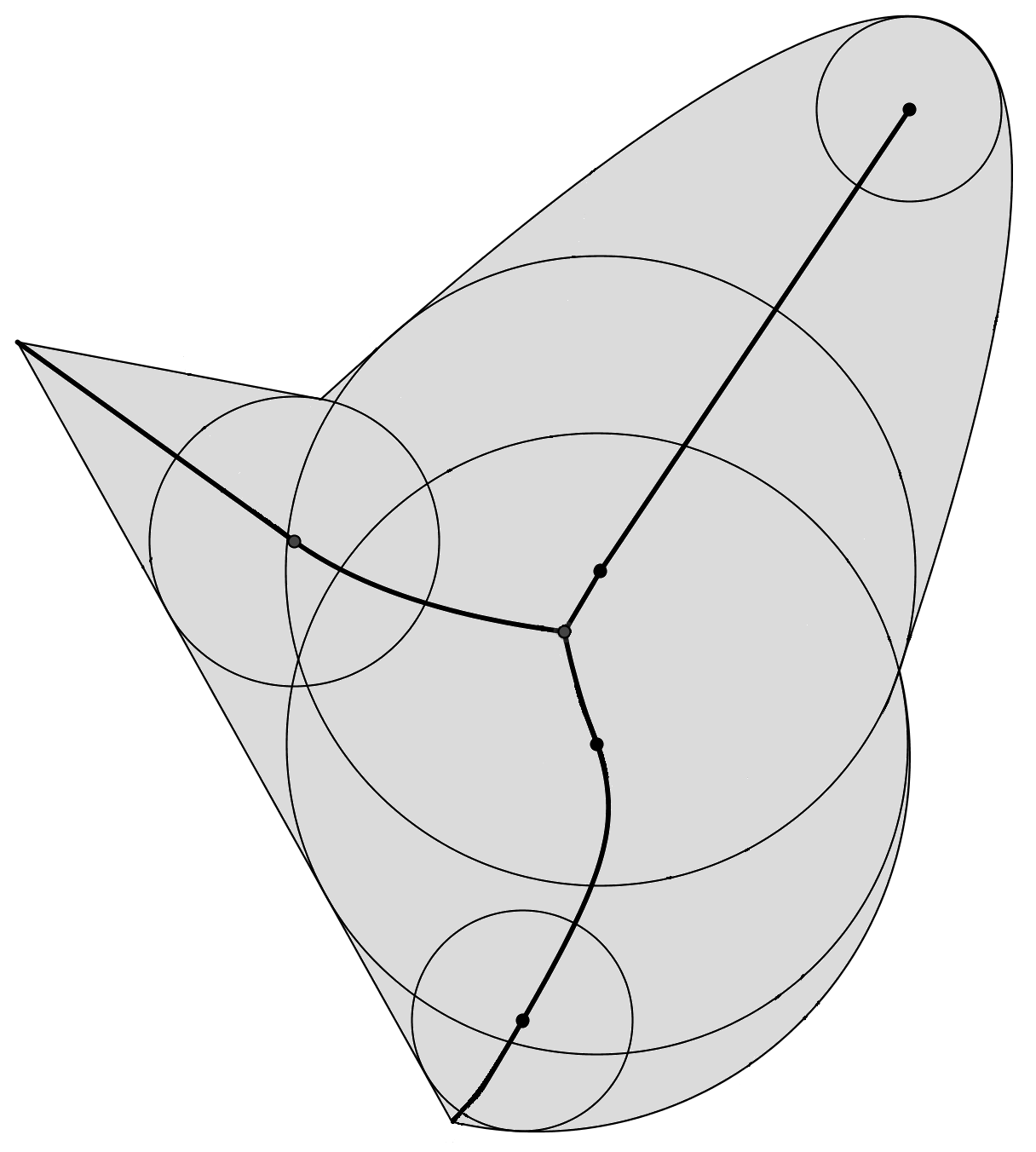}
\caption{The central set of the shaded region consists of three simple curves having a common endpoint}\label{pic}
\end{figure}

\subsection{Counterexamples with centers in the central set}
We start with an observation that can be loosely formulated in the following way: if there exists a counterexample to the Kneser-Poulsen conjecture, then this counterexample can be chosen so that the centers of the balls before the rearrangement lie in the central set of their union. 

For a more accurate statement we first introduce some additional notation: 
for a continuous map $f\colon C_U\to M$ that rearranges the centers of the maximal balls in $U$, we denote by $U_f\subset M$ the union of these maximal balls after their rearrangement by the map $f$. Now we can state the above observation in a more precise form:

\begin{lemma}\label{cut_locus_counterexample_lemma}
Let $U\subset M$ be the union of finitely many closed balls and let $V\subset M$ be the union of the same balls after a contractive rearrangement. Let $f\colon U\to M$ be an arbitrary contraction (1-Lipschitz map), such that its restriction to the centers of the balls provides the considered contractive rearrangement. Then the set $U_f$ is compact (hence has a well-defined volume) and $V\subseteq U_f$. In particular, if 
$\Vol[U_f]\le\Vol[U]$, then $\Vol[V]\le\Vol[U]$.
\end{lemma}

The proof of Lemma~\ref{cut_locus_counterexample_lemma} immediately follows from Lemma~\ref{cut_locus_optimal_lemma} and Lemma~\ref{compact_U_Xf_lemma} proved in Section~\ref{cut_locus_sec}.

We remark that the existence of a contraction $f\colon U\to M$ that satisfies the conditions of Lemma~\ref{cut_locus_counterexample_lemma}, is guaranteed by the Kirszbraun-Valentine extension theorem~\cite{Valentine}.

\subsection{Structure of the central set}
According to Lemma~\ref{cut_locus_counterexample_lemma}, in order to prove the Kneser-Poulsen conjecture, it is enough to rule out the counterexamples with (possibly infinitely many) balls, whose centers in the initial configuration lie in the central set of their union. In order to approach this problem, a better understanding of the structure of the central set is required.
The following lemma gives a combinatorial and topological description of the central set of the union of finitely many balls.

\begin{definition}
	We will say that a set $U\subsetneq M$ is an $n$-dimensional ball-polytope, if it can be represented as a union of finitely many closed balls of positive radii, and $\partial U$ is a 
	codimension one topological submanifold of $M$.
\end{definition}

\begin{lemma}[Structure lemma]\label{cel_structure_lemma}
Let $M$ be one of the three spaces $\bbS^n$, $\bbE^n$ or $\bbH^n$, where $n\ge 2$, and 
let $U\subsetneq M$ be a ball-polytope. Then the following holds:

(i) the central set $C_U$ has a structure of a finite $(n-1)$-dimensional cell complex whose $k$-dimensional cells are $k$-dimensional convex polytopes; 

(ii)  the set $U$ can be represented as the union of finitely many balls with centers at the $0$-dimensional cells of the cell complex $C_U$;

(iii) the sets $C_U$ and $U$ are homotopy equivalent.
\end{lemma}

We prove Lemma~\ref{cel_structure_lemma} in Section~\ref{cut_locus_sec}.

\begin{remark}
In case if $M=\bbE^n$, there is a strong relation between the central set $C_U$ of a ball-polytope $U$ and the dual complex $\mathscr K(U)$ (or dual shape $\mathscr S(U)$) defined in~\cite{Edelsbrunner}. More specifically, let $U$ be the union of finitely many $n$-dimensional balls $B_{x_1},\dots,B_{x_m}\subset \bbE^n$ with corresponding centers $x_1,\dots,x_m\in \bbE^n$ and let $V_{x_1},\dots,V_{x_m}$ be the associated weighted Voronoi cells. We recall that by definition, the convex hull of a subset $I\subset\{x_1,\dots,x_m\}$ is a cell of the dual complex $\mathscr K(U)$, if and only if the dimension of this convex hull is equal to $|I|-1$ and the intersection of truncated Voronoi cells $\bigcap_{x\in I}(V_x\cap B_x)$ is nonempty. The dual shape $\mathscr S(U)$ is defined as the union of all cells from $\mathscr K(U)$. It can be observed from the proof of Lemma~\ref{cel_structure_lemma} that for a ball-polytope $U\subset\bbE^n$, the equality $C_U=\mathscr S(U)$ holds, if and only if $\dim(\mathscr K(U))\le n-1$, while in the general case, the central set $C_U$ is a deformation retract of $\mathscr S(U)$.
\end{remark}

\subsection{Divide and conquer principle}
Let $X\subset C_U$ be a closed non-empty subset. We denote by $U_X$ the union of all maximal balls in $U$ that are centered at points of $X$. 

Now we are ready to formulate the lemma that is crucial for the proposed method. 

\begin{lemma}[Splitting lemma]\label{crucial_lemma}
Let $M$ and $U\subsetneq M$ be the same as in Lemma~\ref{cel_structure_lemma}. Let $X,Y\subset C_U$ be closed non-empty sets, such that $X\cup Y=C_U$. Then
\begin{equation*}
U_X\cap U_Y= U_{X\cap Y}.
\end{equation*}
\end{lemma}
It is worth mentioning that the inclusion $U_{X\cap Y}\subset U_X\cap U_Y$ is trivial, while the opposite inclusion is not obvious.

The proof of Lemma~\ref{crucial_lemma} relies on the properties of a new object that we call a \textit{relative central set}. We introduce this notion in Section~\ref{rel_cut_loc_sec} for a more general class of compact sets in an arbitrary complete Riemannian manifold.

Lemma~\ref{crucial_lemma} is called the Splitting lemma, because it provides the groundwork for applying the ``divide and conquer'' principle and eventually constructing the inductive argument. In order to proceed with the ``divide and conquer'' principle, we have to introduce some additional notation and formulate a technical lemma.

For a closed set $X\subset U$ and a contraction $f\colon C_U\to M$ that rearranges the centers of the maximal balls in $U$, we let $U_{X,f}\subset U_f$ represent the union of the balls that are obtained from 
all maximal balls in $U$ with centers in $X$ after these balls are rearranged by the map~$f$. Lemma~\ref{compact_U_Xf_lemma} proved in Section~\ref{cut_locus_sec}, immediately implies the following statement:

\begin{lemma}\label{U_X_compact_lemma}
Let $M$ and $U\subsetneq M$ be the same as in Lemma~\ref{cel_structure_lemma}, and let $f\colon C_U\to M$ be a contraction. Then for any closed non-empty subset $X\subset C_U$, the sets $U_X$ and $U_{X,f}$ are compact, and in particular, have a well-defined volume.
\end{lemma}

We split the configuration of all maximal balls in $U$ into two smaller configurations, one with centers in $X$ and another -- with centers in $Y$, so that $X\cup Y=C_U$. Then Lemma~\ref{crucial_lemma} and Lemma~\ref{U_X_compact_lemma} imply that the volume of the set $U$ can be represented as
\begin{equation}\label{before contr_eq}
\Vol[U]=\Vol[U_X]+\Vol[U_Y]-\Vol[U_{X\cap Y}].
\end{equation}
An obvious inclusion
$U_{X\cap Y,f}\subset U_{X,f}\cap U_{Y,f}$ implies that
\begin{equation}\label{after_contr_eq}
\Vol[U_f]\le\Vol[U_{X,f}]+\Vol[U_{Y,f}]-\Vol[U_{X\cap Y,f}].
\end{equation}

In the following proposition we formulate a version of the ``divide and conquer'' principle, which is sufficient for a proof of Theorem~\ref{main_theorem}.
\begin{proposition}\label{divide_and_conquer_proposition}
Let $U\subsetneq M$ and $X,Y\subset C_U$ be the same as in Lemma~\ref{crucial_lemma}, and let $f\colon C_U\to M$ be a contraction. If $\Vol[U_{X,f}]\le\Vol[U_X]$, $\Vol[U_{Y,f}]\le\Vol[U_Y]$ and $\Vol[U_{X\cap Y}]=\Vol[U_{X\cap Y,f}]$, then
\begin{equation*}
\Vol[U]\ge\Vol[U_f].
\end{equation*}
\end{proposition}
\begin{proof}
The proof immediately follows from~(\ref{before contr_eq}) and~(\ref{after_contr_eq}). 
\end{proof}

We finish this subsection by formulating another technical lemma whose proof will be given in Section~\ref{cut_locus_sec}. 

\begin{lemma}\label{sub_cutlocus_lemma}
Let $M$ and $U\subsetneq M$ be the same as in Lemma~\ref{cel_structure_lemma}. Let $X\subset C_U$ be a closed non-empty set and let $Z\subset X$ be the relative boundary of $X$ in $C_U$. If $C_{U_{Z}}\subset X$, 
then $C_{U_X}=X$.
\end{lemma}
We note that in the above lemma the inclusion $X\subset C_{U_X}$ is obvious, while the opposite inclusion is not. Lemma~\ref{sub_cutlocus_lemma} will be used in the proof of Theorem~\ref{main_theorem} in the next subsection.

\subsection{Piecewise isometries}
In order to construct an inductive argument, it is convenient to work with contractions that can be described combinatorially. We choose piecewise isometries as a class of such contractions.

\begin{definition}
	A map $f\colon M\to M$ is called a \textit{piecewise isometry}, if the map $f$ is continuous and there exists a locally finite triangulation of $M$, such that for any simplex $T$ of that triangulation, the restriction of $f$ to $T$ is an isometry.
\end{definition}

An important ingredient of the proof of Theorem~\ref{main_theorem} is the following extension theorem:
\begin{theorem}\label{Akopyan_theorem}\cite{Akopyan, Brehm}
	Let $M$ be one of the three spaces $\bbE^n$, $\bbS^n$, $\bbH^n$, where $n\in\bbN$. Then any contraction $f\colon X\to M$ of a finite set $X\subset M$ can be extended to a piecewise isometry on $M$.
\end{theorem}

\begin{proof}[Proof of Theorem~\ref{main_theorem}]
Let $M$ be one of the three spaces $\bbE^2$, $\bbS^2$ or $\bbH^2$.
Assume that there exists a counterexample to the statement of Theorem~\ref{main_theorem}. According to Theorem~\ref{Akopyan_theorem}, the corresponding rearrangement of the centers of the disks can be extended to a piecewise isometry $f\colon M\to M$. We fix this piecewise isometry $f$ together with the associated triangulation of the space $M$. Without loss of generality we may assume that every interior point of the union of the disks in the initial configuration is also an interior point of some disk. (If an interior point is not in the interior of some disk, then it lies in the intersection of at least two boundary circles. There are finitely many such points and we can always place new small disks that cover them.)
We decrease the radii of the disks slightly, if necessary, so that the union of the disks in the initial configuration of the counterexample is a simply connected ball-polytope. 
Let us denote this ball-polytope by $U\subset M$. Since the perturbation of the radii can be arbitrarily small, we may assume that the new configuration of disks still provides a counterexample to the statement of Theorem~\ref{main_theorem}. Then Lemma~\ref{cut_locus_counterexample_lemma} implies that
\begin{equation}\label{mathcal_U_eq}
\Vol[U_f]>\Vol[U].
\end{equation}

Let $\mathcal U$ be the set of all simply connected ball-polytopes $U\subset M$, such that the inequality~(\ref{mathcal_U_eq}) holds.
As we just noticed, the set $\mathcal U$ is non-empty. It follows from Lemma~\ref{cel_structure_lemma} that for every $U\in\mathcal U$, the central set $C_U$ has the structure of a tree with straight edges. Every edge intersects only finitely many simplexes from the triangulation associated to $f$, so it can be split into finitely many edges, such that $f$ restricted to each one of them is an isometry.

For $U\in\mathcal U$, let $\Gamma(U)$ denote the tree structure on $C_U$, such that $f$ restricted to every edge of $\Gamma(U)$ is an isometry, and $\Gamma(U)$ has the smallest possible number of edges. Let $|\Gamma(U)|\ge 0$ denote the number of edges in the graph $\Gamma(U)$.

We choose an element $U\in\mathcal U$ with the minimal number $|\Gamma(U)|$. This number is strictly positive, since otherwise $U$ is a ball and the inequality~(\ref{mathcal_U_eq}) does not hold. Now since $|\Gamma(U)|>0$ and $\Gamma(U)$ is a tree, it has an edge with a vertex of index $1$. Let us denote this edge by $Y\subset C_U$ and define $X=\overline{C_U\setminus Y}$. Since $X\cap Y$ is a singleton, the sets $U_{X\cap Y}$ and $U_{X\cap Y,f}$ are congruent balls, hence their volumes are equal, and since $f$ is an isometry on $Y$, it follows from Proposition~\ref{divide_and_conquer_proposition} that
\begin{equation*}
\Vol[U_{X,f}]>\Vol[U_X].
\end{equation*}
Moreover, it follows from part (ii) of Lemma~\ref{cel_structure_lemma} and from Lemma~\ref{crucial_lemma} that the set $U\setminus U_X$ has a form $B_1\setminus B_2$, where $B_1,B_2\subset M$ are two closed disks centered at the vertexes of the edge $Y$. The latter implies that $U_X$ is a simply connected ball-polytope, hence $U_X\in\mathcal U$.

Finally, since relative boundary of $X$ in $C_U$ is the set $X\cap Y$ that consists of one point, we have $C_{U_{X\cap Y}}=X\cap Y\subset X$, and Lemma~\ref{sub_cutlocus_lemma} implies that $C_{U_X}=X$. Thus $|\Gamma(U_X)|<|\Gamma(U)|$, which contradicts the choice of the ball-polytope $U$. 
\end{proof}

\subsection{The general case of the Kneser-Poulsen conjecture}

As seen from the proof of Theorem~\ref{main_theorem}, in order to apply the same methods to prove the Kneser-Poulsen conjecture in full generality in spaces of arbitrary dimension, one requires a more delicate statement than Proposition~\ref{divide_and_conquer_proposition}, since we cannot always find a splitting of the central set $C_U$ into two pieces $X$ and $Y$ so that $\Vol[U_{X\cap Y}]=\Vol[U_{X\cap Y,f}]$ and neither of the sets $U_X$ and $U_Y$ is equal to $U$. This is the main obstruction. 
However, if we can show that there always exists a splitting $C_U=X\cup Y$, such that
\begin{equation*}
\Vol[U_{X}]+\Vol[U_{Y}]-\Vol[U_{X,f}]-\Vol[U_{Y,f}]\ge \Vol[U_{X\cap Y}]-\Vol[U_{X\cap Y,f}],
\end{equation*}
then~(\ref{before contr_eq}) and~(\ref{after_contr_eq}) will imply that 
\begin{equation*}
\Vol[U]\ge\Vol[U_f].
\end{equation*}

\section{Notation}

In the remaining part of the paper we will always assume that $M$ is a (smooth) complete connected Riemannian manifold, unless otherwise specified.
For a point $p\in M$ and a positive real number $r>0$, let $B_p(r)\subset M$ represent a closed geodesic ball of radius $r$, centered at $p$. By convention, $B_p(0)=\{p\}$.

Let $X$ be a set and consider two functions, $\bfi\colon X\to M$ and $\bfr\colon X\to\bbR^+$. We denote by $B_{X}(\bfi, \bfr)\subset M$
the union of geodesic balls
\begin{equation*}
B_{X}(\bfi,\bfr)=\bigcup_{x\in X} B_{\bfi(x)}(\bfr(x)),
\end{equation*}
where $x$ runs over all elements of $X$. If $X$ is already a subset of $M$ and the map $\bfi$ is the identity map on $X$, in order to shorten the notation, we will write $B_X(\bfr)$ instead of $B_{X}(\bfi,\bfr)$. 

Let $d(\cdot,\cdot)$ be the intrinsic metric on $M$ induced by the Riemannian metric tensor. Consider a set $X\subset M$. A map $f\colon X\to M$ is called a \textit{contraction}, if for every pair of points $x,y\in X$, the inequality $d(x,y)\ge d(f(x),f(y))$ holds.

We also let $\Vol[X]$ represent the volume of a set $X\subset M$ provided that it is defined.

\section{Central sets}\label{cut_locus_sec}

The goal of this section is to establish some important properties of central sets and prove the lemmas formulated in Section~\ref{method_sec}. In particular, we introduce the notion of a \textit{relative central set} and prove that under certain conditions, satisfied by ball-polytopes in $\bbE^n$, $\bbS^n$ or $\bbH^n$, the relative central sets are contractible (c.f. Lemma~\ref{relative_lemma} for a precise statement). For the sake of possible further generalizations, most of the results about the central sets are proved for subsets of an arbitrary complete connected Riemannian manifold. 

We start by giving the definition of the central set of a closed region in a Riemannian manifold $M$.

\begin{definition}\label{max_ball_def}
Let $U\subsetneq M$ be a closed subset of $M$. A closed geodesic ball $B=B_x(r)\subset U$ is called a \textit{maximal ball} in $U$, if it is not a proper subset of any other closed geodesic ball $B'=B_y(R)\subset U$ contained in $U$ and such that $d(x,y)+r\le R$.
\end{definition}

\begin{definition}
Let $U\subset M$ be a closed subset of $M$. A point $p\in U$ is a \textit{cut point} of $U$, if it is a center of a maximal ball in $U$.
\end{definition}

\begin{definition}
The \textit{central set} of a closed set $U$ in $M$ is the set of all cut points of $U$. We denote the central set of $U$ by $C_U$. By $\bfr_U$ we denote the map $\bfr_U\colon C_U\to\bbR^+$, such that for every $x\in C_U$, the ball $B_x(\bfr_U(x))$ is a maximal ball in $U$.
\end{definition}

The following simple lemma is a generalized version of Lemma~\ref{cut_locus_counterexample_lemma} formulated in Section~\ref{method_sec}.

\begin{lemma}\label{cut_locus_optimal_lemma}
Assume that a set $X\subset M$ and a function $\bfr\colon X\to\bbR^+$ are such that the set $U=B_X(\bfr)$ is compact in $M$. If a map $f\colon U\to M$ is a contraction, then
$B_{X}(f,\bfr) \subset B_{C_U}(f,\bfr_U)$.
\end{lemma}
\begin{proof}
Due to compactness of $U$, for every point $p\in X$, the ball $B_p(\bfr(p))$ is contained in some maximal ball $B_q(\bfr_U(q))$, where $q\in C_U$ is a point from the central set, and $d(p,q)+\bfr(p)\le\bfr_U(q)$. Since the map $f$ contracts the distance between centers of the balls, we have $B_{f(p)}(\bfr(p))\subset B_{f(q)}(\bfr_U(q))$, which implies the statement of the lemma.
\end{proof}

Let us introduce the following notation: let $U\subset M$ be a closed set and let $X\subset C_U$ be a closed non-empty subset of $C_U$. For a contraction $f\colon C_U\to M$, we denote by $U_{X,f}\subset M$ the set $U_{X,f}=B_X(f,\bfr_U)$. If $f=\mathrm{id}$, then we abbreviate the notation by writing $U_X$ instead of $U_{X,f}$. 

 The following basic lemma immediately implies Lemma~\ref{U_X_compact_lemma} from Section~\ref{method_sec}.

\begin{lemma}\label{compact_U_Xf_lemma}
Let $U\subset M$ be a compact set and let $X\subset C_U$ be a closed non-empty set. Then for any contraction $f\colon C_U\to M$, the set $U_{X,f}$ is compact.
\end{lemma}
\begin{proof}
Assume that there exists a point $p\in\partial U_{X,f}$, such that $p\not\in U_{X,f}$. Then there exists a sequence of points $x_1,x_2,\dots$ in $X$, such that 
\begin{equation}\label{1_k_inequality}
d(f(x_k),p)-\bfr_U(x_k)<1/k, \quad\text{for all } k\in\bbN.
\end{equation}
From compactness of the set $X$ it follows that the sequence $\{x_k\}$ has a limit point $x\in X$. Since $p\not\in U_{X,f}$, the number $\varepsilon=d(f(x),p)-\bfr_U(x)$ is strictly positive. Finally, since $B_x(\bfr_U(x))$ is a maximal ball in $U$ and $f$ is a contraction, this implies that for every $y\in X$, such that $d(x,y)<\varepsilon/4$, we have $B_{f(y)}(\bfr_U(y))\subset B_{f(x)}(\bfr_U(x)+\varepsilon/2)$, which means that there are infinitely many values of $k$, such that 
\begin{equation*}
d(f(x_k),p)-\bfr_U(x_k)\ge\varepsilon/2.
\end{equation*}
Thus, we arrive at a contradiction with~(\ref{1_k_inequality}).
\end{proof}

\subsection{Central set of a compact region with smooth boundary}\label{cut_l_smooth_sec}
From now on, let $U\subset M$ be a compact connected subset of $M$, such that $S=\partial U$ is a $C^1$-smooth codimension one submanifold of $M$. 
Below we outline some basic facts about the central set. A more rigorous exposition can be found for example in~\cite{Chavel_2006}. We note that in the case when $S$ is a smooth connected submanifold of $M$, the notion of the central set $C_U$ is strongly related to the notion of the \textit{cut locus} of $S$. More precisely, $C_U$ is the connected component of the cut locus of $S$ that is contained in $U$.

We let $NS$ denote the unit normal bundle to the submanifold $S$.
Every element $n\in {NS}$ of the unit normal bundle defines a geodesic on the manifold $M$. Starting from the base point $p\in S$ of the vector $n$ and moving along this geodesic towards the interior of $U$, we must eventually hit the central set $C_U$ at some point $q\in C_U$, since otherwise a compact set $U$ would contain balls of arbitrarily large radius. We denote by $\gamma_n\subset U$ the closed segment of the above geodesic that is contained in $U$ and whose endpoints are $p$ and $q$. Let $l_n$ be the length of $\gamma_n$.  

Let $d_U\colon U\to\bbR^+$ denote the ``distance to the boundary'' function. Namely, $d_U(p)$ is the distance from a point $p\in U$ to the submanifold $S$. It is well known that the distance $d_U(p)$ must be obtained along some geodesic segment $\gamma_n$ that passes through the point $p$. In particular, this implies that 
$
\bigcup_{n\in {NS}}\gamma_n=U
$.
At the same time it is obvious from the definition of the central set that every point $p\in U\setminus C_U$ belongs to exactly one geodesic segment $\gamma_n$.
For $p\in U\setminus C_U$ and $t\in[0,1]$ we define $p_t\in U$ to be the point on this geodesic segment $\gamma_n$, such that $p_t$ is distance $d_U(p)+t(l_n-d_U(p))$ away from $S$.

We define the function $F_U\colon U\times[0,1]\to U$ in the following way: 
\begin{equation*}
F_U(p,t)=
\begin{cases}
p,& \text{if } p\in C_U\\
p_t,&  \text{if } p\in U\setminus C_U.
\end{cases}
\end{equation*}


The following statement was essentially proved in Theorem~2 of~\cite{Milman_Waksman}.
\begin{lemma}\label{retract_lemma}
Let the boundary $S=\partial U$ of a compact subset $U\subset M$ be a $C^1$-smooth codimension one submanifold of $M$. Then the central set $C_U\subset M$ is compact, if and only if the map $F_U$ is a continuous deformation retraction of $U$ onto $C_U$. 
\end{lemma}

\begin{remark}\label{cont_remark}
It is shown in the same paper~\cite{Milman_Waksman} that if the boundary $S=\partial U$ is $C^1$-smooth, then the central set $C_U$ does not have to be compact, while $C^2$-smoothness of $S$ implies compactness of $C_U$. Even if the boundary $S$ is $C^\infty$-smooth, the central set can be quite wild. For example, it might be nontriangulable~\cite{Gluck_Singer_78}. On the other hand, in~\cite{Itoh_Tanaka} it is shown that the function $\rho\colon S\to\bbR^+$ defined by $\rho(p)=d(p,F_U(p,1))$, is Lipschitz, which implies that the canonical interior metric $\delta$ can be introduced on $C_U$, so that $(C_U,\delta)$ is a locally compact and complete length space.
\end{remark}

\subsection{Relative central set}\label{rel_cut_loc_sec}
The following definition seems to be new, as we were not able to find it in the literature.
\begin{definition}[Relative central set]
Let $U\subset M$ be a closed set, and let $C_U$ be the central set of $U$. For a point $p\in U$, we will denote by $C_{U,p}\subset C_U$ \textit{the central set of} $U$ \textit{relative to} $p$ which is defined as the set of all points $x\in C_U$, such that $p\in B_x(\bfr_U(x))$.
\end{definition}

We will also need the following definition:
\begin{definition}
Let $U\subset M$ be a closed set. A point $p\in U$ is called \textit{a standard point} in $U$, if for every point $q\in U$ that can be connected with $p$ by two distinct geodesic segments of lengths $l_1$ and $l_2$ respectively, the closed ball $B_q(\max\{l_1,l_2\})$ is not contained in $U$.
\end{definition}

Now we will assume that $U\subset M$ is a compact connected subset of $M$, such that $S=\partial U$ is a $C^1$-smooth codimension one submanifold of~$M$. 
The following lemma states that if the central set $C_U$ is compact, then the relative central set $C_{U,p}$ of a standard point in $U$ is always contractible. This is essentially the key lemma for the proof of Theorem~\ref{main_theorem}.

\begin{lemma}\label{relative_lemma}
Let $U\subset M$ be a compact connected subset, such that $S=\partial U$ is a $C^1$-smooth codimension one submanifold of $M$, and the central set $C_U$ is compact. 
Let $p\in U$ be a standard point in $U$. 
Then the set $C_{U,p}$ is compact, non-empty and homotopy equivalent to a point.
\end{lemma}
\begin{proof}
We denote by $D_p\subset U$ the set of all points $x\in U$, such that the closed geodesic ball $B_x(d(x,p))$ is contained in $U$. It is easy to verify that $D_p$ is a closed subset of $M$ and $C_{U,p}=D_p\cap C_U$.

Another description of the set $D_p$ can be given in the following way:
for every vector $v\in T_pM$ with $\|v\|=1$, consider the maximal closed geodesic segment $\xi_v\subset U$ starting at $p$ and going in the direction of the vector $v$, such that for every point $x\in\xi_v$, the closed geodesic ball $B_x(d(x,p))$ is contained in $U$. Since $p$ is a standard point in $U$, it follows that $\xi_v\cap\xi_u=\{p\}$, for every pair of distinct unit vectors $u,v\in T_pM$. Moreover, from this observation and from the construction of geodesic segments $\xi_v$ it follows that 
\begin{equation*}
D_p=\bigcup_{v\in T_pM,\\ \|v\|=1}\xi_v,
\end{equation*}
and $D_p$ is homotopy equivalent to a point.

Finally, we observe that if two points $x,y$ belong to the same geodesic segment $\gamma_n$ defined in Subsection~\ref{cut_l_smooth_sec}, and $d_U(x)>d_U(y)$, then $B_y(d_U(y))\subset B_x(d_U(x))$. This implies that if $p\in B_y(d_U(y))$, then also $p\in B_x(d_U(x))$. This means that if $y\in D_p$, then also $x\in D_p$. From this and from Lemma~\ref{retract_lemma} we conclude that the map $F_U$ restricted to the set $D_p\times[0,1]$ is a continuous deformation retraction of $D_p$ onto $C_{U,p}$, hence $C_{U,p}$ is homotopy equivalent to a point.
\end{proof}

\begin{remark}
If the point $p$ is not standard in the set $U$, then the result of Lemma~\ref{relative_lemma} might not hold. For example, let $M$ be the cylinder $\bbC\slash\bbZ$ and let $U\subset M$ be the set $\{z\in\bbC\mid|\Imm z|\le 1\}\slash\bbZ$. Then the central set of $U$ as well as the relative central set of the point zero is the unit circle $\bbR\slash\bbZ$ which is not contractible.
\end{remark}

\begin{corollary}\label{relative_corollary_1}
Let $U\subset M$ be the same as in Lemma~\ref{relative_lemma} and assume that every point in $U$ is standard. Let $X,Y\subset C_U$ be closed non-empty sets, such that $X\cup Y=C_U$. Then 
\begin{equation*}
U_X\cap U_Y=U_{X\cap Y}.
\end{equation*}
\end{corollary}
\begin{proof}
The inclusion $U_{X\cap Y}\subset U_X\cap U_Y$ is obvious. In order to prove that $U_X\cap U_Y \subset U_{X\cap Y}$ we notice that
if a point $p\in U$ belongs both to $U_X$ and $U_Y$, then the relative central set $C_{U,p}$ has a non-empty intersection both with $X$ and with $Y$. Since according to Lemma~\ref{relative_lemma}, the set $C_{U,p}$ is connected, the latter implies that $C_{U,p}$ has a non-empty intersection with $X\cap Y$, hence $p\in U_{X\cap Y}$.
\end{proof}

\subsection{Compact sets with non-smooth boundary}
In order to prove Lemma~\ref{cel_structure_lemma}, Lemma~\ref{crucial_lemma} and Lemma~\ref{sub_cutlocus_lemma} we have to deal with compact sets whose boundary is non-smooth. 
In this subsection we will assume that $U\subset M$ is a compact connected set with a non-empty interior, such that $S=\partial U$ is a locally flat codimension one topological submanifold of $M$, and the radius of every maximal ball in $U$ is greater than some $\varepsilon_0>0$. 
Then for any $\varepsilon>0$, such that $\varepsilon<\varepsilon_0$, the function $\bfr_\varepsilon\colon C_U\to\bbR^+$ defined by the relation
\begin{equation*}
\bfr_\varepsilon = \bfr_U-\varepsilon\cdot\mathbf 1_{C_U},
\end{equation*}
(where $\mathbf 1_{C_U}$ denotes the constant $1$ function on $C_U$) takes positive values and the set 
\begin{equation*}
U^\varepsilon=B_{C_U}(\bfr_\varepsilon)
\end{equation*}
is compact, connected, and has a non-empty interior.

\begin{proposition}\label{smooth_approximation_prop}
For any positive $\varepsilon<\varepsilon_0$, the boundary $\partial U^\varepsilon$ is a $C^1$-smooth codimension one submanifold of $M$.
\end{proposition}
\begin{proof}
Fix a positive real number $\varepsilon_1>0$, such that 
\begin{equation}\label{epsilon_choice_eq}
\varepsilon+\varepsilon_1<\varepsilon_0.
\end{equation}
Consider an arbitrary point $p\in\partial U^\varepsilon$. Since the point $p$ is distance $\varepsilon<\varepsilon_0$ away from $\partial U$, we have $p\not\in C_U$. This implies that there exists a unique point $q\in\partial U$ and a unique geodesic $\gamma\subset U$  passing through $p$ and $q$, such that the length of its segment between $p$ and $q$ is minimal and equal to $\varepsilon$. 

Since $q\in\partial U$, there exists a unique maximal ball $B$ in $U$, such that $B_p(\varepsilon)\subset B$. 
Let $q_1\in \gamma$ be the point on $\gamma$ that is distance $\varepsilon_1$ away from $p$ and lies on the other side from $q$. 
It follows from~(\ref{epsilon_choice_eq}) that $B_{q_1}(\varepsilon+\varepsilon_1)\subset B$, which implies that $d(q_1,\partial U)=\varepsilon+\varepsilon_1$. From this we conclude that $B_{q_1}(\varepsilon_1)\subset U^\varepsilon$.

Finally, we observe that if $q_2\in \gamma$ is the midpoint of the geodesic segment of $\gamma$ between $p$ and $q$, then $B_{q_2}(\varepsilon/2)\cap\partial U^\varepsilon=\{p\}$. Thus, $\partial U^\varepsilon$ is ``squeezed'' between the two balls $B_{q_1}(\varepsilon_1)$ and $B_{q_2}(\varepsilon/2)$ (i.e., intersection of any two of the sets $B_{q_1}(\varepsilon_1)$, $B_{q_2}(\varepsilon/2)$ and $\partial U^\varepsilon$ is equal to $\{p\}$). Now smoothness of $\partial U^\varepsilon$ at the point $p$ follows.
\end{proof}

Let $\intU\subset U$ denote the interior of $U$. Since $\partial U$ is a locally flat manifold, the collaring theorem of~\cite{Brown_62} implies that the sets $U$ and $\intU$ are homotopy equivalent.
It is easy to verify that $C_{U^\varepsilon}=C_U$, for all sufficiently small $\varepsilon>0$. Moreover, for all sufficiently small $\varepsilon>\delta>0$, the functions $F_{U^\varepsilon}$ and $F_{U^\delta}$ (which are defined because of Proposition~\ref{smooth_approximation_prop}) coincide on their common domain of definition $U^\varepsilon\times [0,1]$, hence the functions $F_{U^\varepsilon}$ converge point-wise to a function 
\begin{equation*}
\accentset{\circ}F_U\colon\intU\times[0,1]\to\intU,
\end{equation*}
as $\varepsilon\to 0^+$.
It follows from Remark~\ref{cont_remark}, that the map $\accentset{\circ}F_U$ is not necessarily continuous. 
On the other hand, the fact that $C_{U^\varepsilon}=C_U$ together with Lemma~\ref{retract_lemma} implies the following statement:

\begin{lemma}\label{retract_lemma2}
Let $U\subset M$ be a compact connected set with a non-empty interior, such that $S=\partial U$ is a locally flat codimension one topological submanifold of $M$, and every maximal ball in $U$ has positive radius. Then the central set $C_U$ is compact, if and only if the map $\accentset{\circ}F_U$ is a continuous deformation retraction of $U$ onto $C_U$.
\end{lemma}

The following lemma is an analog of Corollary~\ref{relative_corollary_1}.

\begin{lemma}\label{rel_intersect_lemma}
Let $U\subset M$ be a compact connected set with a non-empty interior, such that the following holds:
\begin{itemize}
   \item $S=\partial U$ is a locally flat codimension one topological submanifold of~$M$;
   \item every maximal ball in $U$ has positive radius;
   \item the central set $C_U$ is compact;
   \item every point in $U$ is standard.
\end{itemize}
Let $X,Y\subset C_U$ be closed non-empty sets, such that $X\cup Y=C_U$. Then
\begin{equation*}
U_X\cap U_Y=U_{X\cap Y}.
\end{equation*}
\end{lemma}
\begin{proof}
The inclusion $U_{X\cap Y}\subset U_X\cap U_Y$ is obvious. Let us prove the opposite inclusion. First, assume that $p\in U_X\cap U_Y$ is an interior point of $U_X\cap U_Y$. Then there exists a positive $\varepsilon>0$, such that $p\in U^\varepsilon_X\cap U^\varepsilon_Y$, hence by Corollary~\ref{relative_corollary_1}, we have $p\in U^\varepsilon_{X\cap Y}$, which implies that $p\in U_{X\cap Y}$.

Finally, since all maximal balls in $U$ have positive radius, every point $p\in\partial (U_X\cap U_Y)$ is a limit point of the interior of $U_X\cap U_Y$, hence by the previous case we have $p\in\partial U_{X\cap Y}$. Since according to  Lemma~\ref{compact_U_Xf_lemma}, the set $U_{X\cap Y}$ is closed, this implies that $p\in U_{X\cap Y}$.
\end{proof}

If $X\subset C_U$ is a closed set, Lemma~\ref{compact_U_Xf_lemma} says that the set $U_X$ is compact, hence we can consider its central set which we will denote by $C_{U_X}$. It is easy to see that $X\subset C_{U_X}$, while the opposite inclusion does not necessarily hold.

\begin{lemma}\label{sub_cutlocus_corollary}
Let $U\subset M$ be the same as in Lemma~\ref{rel_intersect_lemma}, and let $Z\subset X\subset C_U$ be closed non-empty sets, such that $Z\subset X$ is the relative boundary of $X$ in $C_U$. If $C_{U_{Z}}\subset X$, 
then $C_{U_X}=X$.
\end{lemma}
\begin{proof}
If a maximal ball in $U$ is contained in $U_X$, then it is also a maximal ball in $U_X$, since $U_X\subset U$. Thus $X\subset C_{U_X}$. This means that in order to prove the lemma, it is enough to show the opposite inclusion.

Define the set $Y\subset C_U$ as the closure $Y=\overline{C_U\setminus X}$. Then $X\cap Y=Z$.
We proceed with a proof by contradiction. Assume that there exists a ball $B\subset U_X$ that is a maximal ball in $U_X$ and whose center is not in $X$. Then there exists a ball $B'\subset U$ centered outside of $X$, such that $B'$ is maximal in $U$ and $B\subset B'$. Since $X\cup Y=C_U$, this implies that the center of the ball $B'$ belongs to $Y$, and $B'\subset U_Y=B_Y(\bfr_U)$. Hence, we have $B\subset U_X\cap U_Y$, and then, according to Lemma~\ref{rel_intersect_lemma}, we have $B\subset U_{Z}$.
Since the ball $B$ is maximal in $U_X$ and $U_{Z}\subset U_X$, the ball $B$ is also maximal in $U_{Z}$, hence, its center belongs to $C_{U_{Z}}$. Finally, since $C_{U_{Z}}\subset X$, we obtain a contradiction. Thus we have $C_{U_X}=X$.
\end{proof}

Now we can give a proof of Lemma~\ref{crucial_lemma} and Lemma~\ref{sub_cutlocus_lemma} modulo the result of Lemma~\ref{cel_structure_lemma}:

\begin{proof}[Proof of Lemma~\ref{crucial_lemma} and Lemma~\ref{sub_cutlocus_lemma}]
We notice that if $M$ is one of the three spaces $\bbE^n$, $\bbS^n$, $\bbH^n$, and $U$ is a proper subset of $M$, then every point in $U$ is standard. Part~(i) of Lemma~\ref{cel_structure_lemma} implies that the central set $C_U$ is compact. Now Lemma~\ref{crucial_lemma} and Lemma~\ref{sub_cutlocus_lemma} follow from Lemma~\ref{rel_intersect_lemma} and Lemma~\ref{sub_cutlocus_corollary} respectively.
\end{proof}

Finally, we pass to the proof of Lemma~\ref{cel_structure_lemma}. 

\begin{proof}[Proof of Lemma~\ref{cel_structure_lemma}]
Let $M$ be one of the three spaces $\bbE^n$, $\bbS^n$ or $\bbH^n$, and 
let $B_1,\dots,B_k\subset M$ be $n$-dimensional closed balls of positive radius, such that
$$U=B_1\cup\dots\cup B_k\subset M$$
is an $n$-dimensional ball-polytope. By definition of a ball-polytope, the boundary $S=\partial U$ is a topological manifold of dimension $n-1$. 

Let $\mathcal C_S$ be a family that consists of all connected components of all intersections of the form
\begin{equation*}
\left(\bigcap_{i\in\alpha}B_i\right)\bigcap S,
\end{equation*}
where $\alpha$ runs over all subsets of the set $\{1,\dots,k\}$. Clearly, $\mathcal C_S$ is a finite family, and
\begin{equation*}
\bigcup_{C\in\mathcal C_S} C=S.
\end{equation*}

A set $C$ from $\mathcal C_S$ will be called \textit{a cell} of $S$. Every cell $C\in\mathcal C_S$ lies in the intersection of finitely many $(n-1)$-dimensional spheres, so if $C$ is not a single point, then there exists a minimal positive integer $m\in\bbN$, for which there is a unique $m$-dimensional sphere in $M$ that contains the cell $C$. We denote this sphere by $S_C\subset M$. For every $m=1,\dots,n-1$ we let $\mathcal C_S^m\subset\mathcal C_S$ represent the set of all $C\in\mathcal C_S$, for which the sphere $S_C$ is $m$-dimensional, and we denote by $\mathcal C_S^0$ the set of all $C\in\mathcal C_S$ that are singletons. Thus, $\mathcal C_S$ is a disjoint union
\begin{equation*}
\mathcal C_S=\bigsqcup_{m=0}^{n-1}\mathcal C_S^m.
\end{equation*}
We will say that a cell $C\in\mathcal C_S$ has dimension $m$, if $C\in\mathcal C_S^m$. 
\begin{definition}
For a cell $C\in\mathcal C_S$, we will say that a point $p\in C$ is an interior point of $C$, if it does not belong to any cell $C'\in\mathcal C_S$ of dimension strictly smaller than $\dim(C)$.
\end{definition}
\noindent In particular, every $0$-dimensional cell $C\in\mathcal C_S^0$ consists of one point that is an interior point of $C$.

For future reference let us formulate the following basic geometric facts:

\begin{proposition}\label{interior_point_prop3}
Every point $p\in S$ is an interior point of some cell $C\in\mathcal C_S$.
\end{proposition}

\begin{proposition}\label{interior_point_prop2}
Let $C\in\mathcal C_S$ be a cell of positive dimension. Then $p\in C$ is an interior point of $C$, if and only if $p$ belongs only to those spheres $\partial B_i$ which contain the sphere $S_C$.
\end{proposition}

Now we prove the following statement:

\begin{proposition}\label{cell_contain_prop}
If the boundary sphere of a maximal ball in $U$ contains an interior point of a positive dimensional cell $C\in\mathcal C_S$, then this sphere contains the whole cell $C$.
\end{proposition}
\begin{proof}
Let $p\in C$ be an interior point of $C$ that also belongs to the boundary of an $n$-dimensional ball $B\subset U$. Assume that $B$ is a maximal ball in $U$. We will show that $S_C\subset\partial B$. We split the proof into two cases:

\textit{Case 1:} We weaken the assumption on the set $U=B_1\cup\dots\cup B_k$ by requiring $S=\partial U$ to be a submanifold only locally at point $p$. At the same time we assume that $p\in \partial B_i$, for all $i\in\{1,\dots,k\}$. Then in particular we have  
\begin{equation}\label{S_C_case1_eq}
C=S_C=\bigcap_{i=1}^k\partial B_i.
\end{equation}
Let $q\neq p$ be another point, where $\partial B$ touches $S$, and let $\alpha\subset\{1,\dots,k\}$ be the set of all indexes $i$, such that $q\in\partial B_i$. Consider the sphere
$$
S_\alpha=\bigcap_{i\in\alpha}\partial B_i.
$$
According to~(\ref{S_C_case1_eq}), we have an inclusion $S_C\subset S_\alpha$, hence, $p$ is a common point of $\partial B$ and $S_\alpha$. Since the spheres $S_\alpha$ and $\partial B$ must also be tangent at the point $q$, and $q\neq p$, this implies that $S_\alpha\subset\partial B$, and in particular, $S_C\subset\partial B$.

\textit{Case 2:} Now we consider the general case. Let $\beta\subset\{1,\dots, k\}$ be the set of all integers $i$, such that $S_C\subset B_i$. 
Consider the set $U_\beta=\cup_{i\in\beta}B_i$. Since $U$ is a ball-polytope, it follows from Proposition~\ref{interior_point_prop2} that $\partial U_\beta$ is a local manifold around the point $p$, and there exists an $n$-dimensional ball $B'\subset B$ that is maximal in $U_\beta$, and whose boundary sphere $\partial B'$ is tangent to $\partial B$ at $p$. Now Case~1 implies that $S_C\subset B'$, and since $B'$ is contained in $B$ and $S_C$ cannot have common points with the interior of $B$, we conclude that $B=B'$ and $S_C\subset B$.
\end{proof}

Let $\Lambda\subset S$ be a finite set, such that for every cell $C\in\mathcal C_S$, the set $\Lambda\cap C$ is non-empty, and for every positive-dimensional cell $C\in\mathcal C_S$, the affine span of $\Lambda\cap C$ has dimension $\dim(C)+1$.
We notice that the existence of such a set $\Lambda$ follows from the fact that the family $\mathcal C_S$ is finite.

Let us formulate the statement that is a straightforward consequence of the assumption on $\Lambda$. 
\begin{proposition}\label{CLambda_contain_prop}
For every cell $C\in\mathcal C_S$, an $(n-1)$-dimensional sphere contains the set $\Lambda\cap C$, if and only if this sphere contains the whole cell $C$.
\end{proposition}

The following proposition immediately implies part (i) of Lemma~\ref{cel_structure_lemma}.

\begin{proposition}
The central set $C_U$ is a subcomplex of the $(n-1)$-skeleton of the Voronoi cell decomposition of the set $\Lambda$.
\end{proposition}
\begin{proof}
First, let us prove that the central set $C_U$ is contained in the $(n-1)$-skeleton of the Voronoi cell decomposition of the set $\Lambda$. In order to do this, it is enough to show that the boundary sphere of every maximal ball in $U$ contains at least two distinct points from $\Lambda$. This can be shown in the following way: the boundary sphere of every maximal ball in $U$ contains at least two distinct points from $S$. If both of these points form $0$-dimensional cells from $\mathcal C_S^0$, then by the assumption on $\Lambda$, these two points belong to $\Lambda$. If at least one of these two points does not belong to a $0$-dimensional cell from $\mathcal C_S$, then, by Proposition~\ref{interior_point_prop3}, it is an interior point of some positive dimensional cell $C\in\mathcal C_S$. In this case, Proposition~\ref{cell_contain_prop} implies that the whole cell $C$ is contained in the boundary sphere of the maximal ball. In particular, according to the assumption on $\Lambda$, this means that at least two distinct points from $\Lambda$ belong to this boundary sphere.

Now we claim that if an interior point of an $m$-dimensional Voronoi cell belongs to $C_U$, then the whole cell is contained in $C_U$ as well. Indeed, every $m$-dimensional Voronoi cell $V$ corresponds to some subset $\Lambda_V\subset\Lambda$ with the property that for every point $p\in V$ there exists an $(n-1)$-dimensional sphere $S_p$ that is centered at $p$ and contains the set $\Lambda_V$, and if $p$ is an interior point of $V$, then the corresponding sphere $S_p$ does not contain any other points from $\Lambda$.
Then Proposition~\ref{CLambda_contain_prop} implies that there exists a collection of cells $\mathcal A_V\subset\mathcal C_S$, such that for every interior point $p\in V$, the sphere $S_p$ contains precisely the cells from $\mathcal A_V$ and no other cells from $\mathcal C_S$.

Assume that there exists an interior point $x$ of the cell $V$, such that $x\in C_U$, and a point $y\in V$, such that $y\not\in C_U$. This is equivalent to the statement that $S_x$ is the boundary sphere of a maximal ball in $U$, but $S_y$ is not. In other words, there are no points of $S$ in the interior of the sphere $S_x$, and there are such points in the interior of the sphere $S_y$. 
We connect the points $x$ and $y$ by a geodesic segment $l\subset V$. Since $V$ is convex, all points of $l$, except possibly the point $y$, lie in the interior of $V$. Since the sphere $S_p$ depends continuously (in Hausdorff topology) on the point $p\in V$, there exists an intermediate point $q\in l$, $q\neq y$, with the property that for all $p\in l$ lying between $x$ and $q$, there are no points of $S$ in the interior of the sphere $S_p$, and for all $p\in l$ sufficiently close to $q$ but lying between $y$ and $q$, such points of $S$ exist. Then continuous dependence of $S_p$ on $p\in V$ implies that the sphere $S_q$ is the boundary sphere of a maximal ball in $U$, and $S_q$ must contain a point $w\in S$ that does not belong to any of the cells from $\mathcal A_V$. According to Proposition~\ref{interior_point_prop3}, there exists a cell $C\in\mathcal C_S$, such that $w\in C$ and $w$ is either an interior point of $C$, or $\dim(C)=0$. In both cases either Proposition~\ref{cell_contain_prop} or the fact that $C$ is a singleton, imply that the cell $C\not\in\mathcal A_V$ is a subset of $S_q$. The latter contradicts to the fact that $q$ is an interior point of $V$. 
\end{proof}

In order to prove part~(ii) of Lemma~\ref{cel_structure_lemma}, we notice that the $(n-1)$-dimensional cells of the family $\mathcal C_S$ cover the whole boundary $S$ of the ball-polytope $U$. This implies that the maximal balls centered at all $0$-dimensional cells of $C_U$, cover the whole set $U$. 

Finally, it follows from part~(i) of Lemma~\ref{cel_structure_lemma} that the central set $C_U$ is compact, hence by Lemma~\ref{retract_lemma2}, the function 
$\accentset{\circ}F_U\colon\intU\times[0,1]\to \intU$  is a continuous deformation retraction of $\intU$ onto $C_U$, which proves part~(iii) of Lemma~\ref{cel_structure_lemma}.
\end{proof}

\bibliography{cut_locus_biblio}
\bibliographystyle{plain}

\end{document}